\documentclass[a4paper,11pt]{amsart}
\usepackage{latexsym}
\usepackage{amsmath,amssymb}
\usepackage{amsmath}

\usepackage[all]{xy}
\usepackage{enumerate}

\newtheorem{theo}{Theorem}[section]
\newtheorem{coll}[theo]{Corollary}
\newtheorem{lemm}[theo]{Lemma}

\newtheorem{defn}[theo]{Definition}

\newtheorem{rem}[theo]{Remark}

\newcommand{\Hom}{{\rm Hom}}

\renewcommand{\hom}[3]{\mathrm{Hom}_{#1}(#2,#3)}
\newcommand{\HOM}[3]{\mathrm{HOM}_{#1}(#2,#3)}

\begin{document}
\sloppy

\title[Rickart and dual Rickart objects]{Rickart and dual Rickart objects in abelian categories: transfer via functors}

\author[S. Crivei]{Septimiu Crivei}

\address{Department of Mathematics, Babe\c s-Bolyai University, Str. M. Kog\u alniceanu 1,
400084 Cluj-Napoca, Romania} \email{crivei@math.ubbcluj.ro}

\author[G. Olteanu]{Gabriela Olteanu}

\address{Department of Statistics-Forecasts-Mathematics, Babe\c s-Bolyai University, Str. T. Mihali 58-60, 400591
Cluj-Napoca, Romania} \email{gabriela.olteanu@econ.ubbcluj.ro}


\begin{abstract} We study the transfer of (dual) relative Rickart properties via functors between abelian categories,
and we deduce the transfer of (dual) relative Baer property. We also give applications to Grothendieck categories,
comodule categories and (graded) module categories, with emphasis on endomorphism rings.
\end{abstract}

\date{December 9, 2017}

\thanks{We would like to thank the referee for comments and suggestions, which improved the presentation of the paper. 
We acknowledge the support of the grant PN-II-ID-PCE-2012-4-0100.}

\maketitle

\section{Introduction}

Rickart and dual Rickart objects in abelian categories were introduced and studied by Crivei and K\"or \cite{CK} with a
twofold motivation. First, they generalize relative regular objects in abelian categories in the sense of D\u asc\u
alescu, N\u ast\u asescu, Tudorache and D\u au\c s \cite{DNT,DNTD,Daus}. 
The primary interest in studying these concept traces back to the work of von Neumann
\cite{vN} on regular rings and Zelmanowitz \cite{Zel} on regular modules. Given two objects $M$ and $N$ of 
an arbitrary category $\mathcal{C}$, $N$ is called \emph{$M$-regular} 
if for every morphism $f:M\to N$, there exists a morphism $g:N\to M$ such that $f=fgf$.
When $\mathcal{A}$ is an abelian category, $N$ is $M$-regular if and only if for every morphism $f:M\to N$, 
${\rm Ker}(f)$ is a direct summand of $M$ and ${\rm Im}(f)$ is a direct summand of $N$ 
\cite[Proposition~3.1]{DNTD}. The main idea from \cite{CK} was to split the study of these two conditions characterizing 
relative regularity. Thus, $N$ is called $M$-Rickart if for every morphism $f:M\to N$, 
${\rm Ker}(f)$ is a direct summand of $M$; also, $N$ is called dual $M$-Rickart 
if for every morphism $f:M\to N$, ${\rm Im}(f)$ is a direct summand of $N$. 
Hence an object of an abelian category is relative regular 
if and only if it is relative Rickart and dual relative Rickart \cite[Corollary~2.3]{CK}. 
Since relative Rickart and dual relative Rickart properties are categorical duals to each other, 
the setting of abelian categories allows one to reduce their investigation to the study of one of them, 
and afterwards to apply the duality principle in order to obtain the corresponding results for the other one.

Secondly, our concepts generalize to the level of abelian categories Rickart and dual Rickart modules 
in the sense of Lee, Rizvi and Roman \cite{LRR10,LRR11,LRR12}, and in particular, Baer and dual Baer modules 
studied by Rizvi and Roman \cite{RR04,RR09} and Keskin T\"ut\"unc\"u, Smith, Toksoy and Tribak \cite{KST,KT}. 
A unified approach of Baer and dual Baer modules via Baer-Galois connections was given by Olteanu in \cite{GO}, 
following the approach by Crivei from \cite{C13}. The root of (dual) Baer and (dual) Rickart modules 
traces back to the work of Kaplansky \cite{K} on Baer rings and Maeda \cite{Maeda} on Rickart rings. 
Given two objects $M$ and $N$ of an abelian category $\mathcal{A}$, $N$ is called $M$-Baer 
if for every family $(f_i)_{i\in I}$ with each $f_i\in \Hom_{\mathcal{A}}(M,N)$, $\bigcap_{i\in I}
{\rm Ker}(f_i)$ is a direct summand of $M$, and dual $M$-Baer if for every family $(f_i)_{i\in I}$ 
with each $f_i\in \Hom_{\mathcal{A}}(M,N)$, $\sum_{i\in I} {\rm Im}(f_i)$ is a direct summand of $N$. 
If there exists the product $N^I$ for every set $I$, then $N$ is $M$-Baer if and only if $N^I$ is $M$-Rickart for every set $I$;
also, if there exists the coproduct $M^{(I)}$ for every set $I$, then $N$ is dual $M$-Baer if and only
if $N$ is dual $M^{(I)}$-Rickart for every set $I$ \cite[Lemma~6.2]{CK}. 
Hence the study of (dual) relative Baer objects can be easily deduced from the study of (dual) relative Rickart objects.

Crivei and K\"or \cite{CK} investigated general properties of (dual) Rickart objects, (co)products of 
(dual) relative Rickart objects and classes all of whose objects are (dual) relative Rickart. 
Their theory was applied to the study of relative regular objects in abelian categories 
in order to obtain some easier proofs and more detailed results, and to the study of (dual) Baer objects in abelian categories, 
whose theory has been naturally developed from the theory of (dual) Rickart objects. Moreover, 
the framework of abelian categories allowed a wide use of the duality principle in order to obtain immediate dualizations, 
and a unified treatment of notions that have been considered separately in the literature. 

The transfer of (dual) relative Rickart properties via (additive) functors between abelian categories as well as their
transfer to endomorphism rings was left to be studied in a separate paper. The aim of the present paper is to complete
this study. We consider the behaviour of (dual) relative Rickart objects and (dual) relative Baer objects under functors
between abelian categories as well as under taking their endomorphism rings in module categories. The statements of our
results usually have two parts, one concerning relative Rickart (Baer) objects and the other one concerning dual
relative Rickart (Baer) objects. Usually we only prove one of them, the other one following by the duality principle in
abelian categories. In general we have considered the case of covariant additive functors, while the contravariant case
may be treated by dualization.

In Section 2, we show that a left exact fully faithful functor preserves and reflects the relative Rickart property. As
a consequence, for a Grothendieck category $\mathcal{A}$ with a generator $U$ and $R={\rm End}_{\mathcal{A}}(U)$, the
functor $S={\rm Hom}_{\mathcal{A}}(U,-):\mathcal{A}\to {\rm Mod}(R)$ from $\mathcal{A}$ to the category ${\rm Mod}(R)$
of right $R$-modules preserves and reflects the relative Rickart property. Other applications are given to module and
comodule categories. We also prove that a Maschke functor reflects the relative Rickart property. Next we consider the
situation of an adjoint pair $(L,R)$ of covariant functors $L:\mathcal{A}\to \mathcal{B}$ and $R:\mathcal{B}\to
\mathcal{A}$ between abelian categories with counit $\varepsilon:LR\to 1_{\mathcal{B}}$. For two objects $M,N\in {\rm
Stat}(R)=\{B\in \mathcal{B}\mid \varepsilon_B \textrm{ is an isomorphism}\}$, we prove that $N$ is $M$-Rickart in
$\mathcal{B}$ if and only if $R(N)$ is $R(M)$-Rickart in $\mathcal{A}$ and for every morphism $f:M\to N$, ${\rm Ker}(f)$
is $M$-cyclic if and only if $R(N)$ is $R(M)$-Rickart in $\mathcal{A}$ and for every morphism $f:M\to N$, ${\rm
Ker}(f)\in {\rm Stat}(R)$. We also discuss the case of an adjoint pair of contravariant functors. 

In Sections 3 and 4 we show how our results on (dual) relative Rickart objects can be used for studying (dual) relative
Baer objects in abelian categories as well as endomorphism rings in (graded) module categories. In Section 3 we recall
the connection between (dual) relative Baer objects and (dual) relative Rickart objects in abelian categories. Then we
may use the results established in Section 2 in order to naturally deduce corresponding properties for (dual) relative
Baer objects.

In Section 4 we discuss the transfer of (dual) self-Rickart and (dual) self-Baer properties to endomorphism rings of
(graded) modules. We derive the following theorem relating the self-Rickart properties of a module and of its
endomorphism ring, which enriches existing results such as \cite[Theorem~3.9]{LRR10}: a right $R$-module $M$ with
$S={\rm End}_R(M)$ is self-Rickart if and only if $S$ is a self-Rickart right $S$-module and for every $f\in S$, ${\rm
Ker}(f)$ is an $M$-cyclic object if and only if $S$ is a self-Rickart right $S$-module and for every $f\in S$, ${\rm
Ker}(f)\in {\rm Stat}({\rm Hom}_R(M,-))$. We also prove that $M$ is self-Rickart if and only if $S$ is a self-Rickart
right $S$-module and for every $f\in S$, ${\rm ker}(f)$ is a locally split monomorphism if and only if $S$ is a
self-Rickart right $S$-module and $M$ is $k$-quasi-retractable. We present corresponding results for self-Baer modules.
We also consider a similar application to categories of graded modules.

\section{(Dual) relative Rickart objects}

Let $\mathcal{A}$ be an abelian category. For every morphism $f:M\to N$ in $\mathcal{A}$ we have the following notation
and \emph{analysis} involving its kernel, cokernel, image and coimage:
$$\SelectTips{cm}{}
\xymatrix{
{\rm Ker}(f) \ar[r]^-{{\rm ker}(f)} & M \ar[r]^f \ar[d]_{{\rm coim}(f)} & N \ar[r]^-{{\rm coker}(f)} & {\rm Coker}(f) \\
& {\rm Coim}(f) \ar[r]_-{\overline{f}} & {\rm Im}(f) \ar[u]_{{\rm im}(f)} &  
}$$
where $\overline{f}$ is an isomorphism. 

Recall that a morphism $f:A\to B$ is called a \emph{section} (or \emph{split monomorphism}) if there is a morphism
$f':B\to A$ such that $f'f=1_A$, and a \emph{retraction} (or \emph{split epimorphism}) if there is a morphism $f':B\to
A$ such that $ff'=1_B$. 

Now let us recall from \cite{CK,DNTD} the definitions of the main concepts of the present paper.

\begin{defn} \rm Let $M$ and $N$ be objects of an abelian category $\mathcal{A}$. Then $N$ is called:
\begin{enumerate}
\item \emph{$M$-Rickart} if the kernel of every morphism $f:M\to N$ is a section, or equivalently, the coimage of every
morphism $f:M\to N$ is a retraction.
\item \emph{dual $M$-Rickart} if the cokernel of every morphism $f:M\to N$ is a retraction, or equivalently, the image
of every morphism $f:M\to N$ is a section.
\item \emph{$M$-regular} if $N$ is $M$-Rickart and dual $M$-Rickart.
\item \emph{self-Rickart} if $N$ is $N$-Rickart.
\item \emph{dual self-Rickart} if $N$ is dual $N$-Rickart.
\item \emph{self-regular} if $N$ is $N$-regular.
\end{enumerate}
\end{defn}

\begin{rem} \rm (a) The definitions of (dual) relative Rickart objects can be reformulated as follows: 
$N$ is called $M$-Rickart if for every morphism $f:M\to N$, ${\rm Ker}(f)$ is a direct summand of $M$, 
and dual $M$-Rickart if for every morphism $f:M\to N$, ${\rm Im}(f)$ is a direct summand of $N$.

(b) Let us point out the difference between the terminology used in the theory of relative regular
objects in \cite{DNTD} and continued by us in \cite{CK} and in the present paper, and the terminology used in the theory
of (dual) relative Rickart modules in \cite{LRR10,LRR11}. 
In the latter the roles of $M$ and $N$ are swapped, so that a module $M$ is called $N$-Rickart
if and only if the kernel of every morphism $f:M\to N$ is a section, and
dual $N$-Rickart if and only if the cokernel of every morphism $f:M\to N$ is a retraction.
Also, our (dual) self-Rickart objects are simply called (dual)
Rickart modules in module categories, while relative regular objects were called relative endoregular in \cite{LRR13}. 
\end{rem}

We begin our study of the transfer of (dual) relative Rickart properties via (additive) functors with the easy case of a
fully faithful covariant functor. 

\begin{theo} \label{t:ff} Let $F:\mathcal{A}\to \mathcal{B}$ be a fully faithful covariant functor between abelian
categories, and let $M$ and $N$ be objects of $\mathcal{A}$. 
\begin{enumerate}
\item Assume that $F$ is left exact. Then $N$ is $M$-Rickart in $\mathcal{A}$ if and only if $F(N)$ is $F(M)$-Rickart
in $\mathcal{B}$.
\item Assume that $F$ is right exact. Then $N$ is dual $M$-Rickart in $\mathcal{A}$ if and only if $F(N)$ is dual
$F(M)$-Rickart in $\mathcal{B}$.
\end{enumerate}
\end{theo}

\begin{proof} 
(1) Assume that $N$ is $M$-Rickart. Let $g:F(M)\to F(N)$ be a morphism in $\mathcal{B}$. Since $F$ is full, we have
$g=F(f)$ for some morphism $f:M\to N$ in $\mathcal{A}$. Since $N$ is $M$-Rickart, $k={\rm ker}(f):K\to M$ is a section.
It follows that $F(k)$ is a section. Since $F$ is left exact, we have ${\rm ker}(g)={\rm ker}(F(f))=F({\rm
ker}(f))=F(k)$. Hence ${\rm ker}(g)=F(k)$ is a section, which shows that $F(N)$ is $F(M)$-Rickart.

Conversely, assume that $F(N)$ is $F(M)$-Rickart. Let $f:M\to N$ be a morphism in $\mathcal{A}$ with kernel $k:K\to M$.
Since $F(N)$ is $F(M)$-Rickart, $l={\rm ker}(F(f)):L\to F(M)$ is a section. Since $F$ is left exact, we have $l={\rm
ker}(F(f))=F({\rm ker}(f))=F(k)$. Since $F$ is fully faithful, it reflects sections. Hence ${\rm ker}(f)=k$ is a
section, which shows that $N$ is $M$-Rickart. 
\end{proof}

For Grothendieck categories we have the following corollary.

\begin{coll} \label{c:gp} Let $\mathcal{A}$ be a Grothendieck category with a generator $U$, $R={\rm
End}_{\mathcal{A}}(U)$, $S={\rm Hom}_{\mathcal{A}}(U,-):\mathcal{A}\to {\rm Mod}(R)$, and let $M$ and $N$ be objects
of $\mathcal{A}$. Then $N$ is an $M$-Rickart object of $\mathcal{A}$ if and only if $S(N)$ is an $S(M)$-Rickart right
$R$-module.
\end{coll}

\begin{proof} By the Gabriel-Popescu Theorem \cite[Chapter~X, Theorem~4.1]{St}, $S$ is a fully faithful functor which
has an exact left adjoint $T:{\rm Mod}(R)\to \mathcal{A}$. Since $S$ is a right adjoint, it is left exact. Then the
conclusion follows by Theorem \ref{t:ff}.  
\end{proof}

By an \emph{adjoint triple} of functors we mean a triple $(L,F,R)$ of covariant functors $F:\mathcal{A}\to \mathcal{B}$
and $L,R:\mathcal{B}\to \mathcal{A}$ such that $(L,F)$ and $(F,R)$ are adjoint pairs of functors. Then $F$ is an exact
functor as a left and right adjoint. Note that $L$ is fully faithful if and only if $R$ is fully faithful
\cite[Lemma~1.3]{DT}. Then Theorem \ref{t:ff} yields the following consequence.  

\begin{coll} \label{c:tripleff} Let $(L,F,R)$ be an adjoint triple of covariant functors $F:\mathcal{A}\to \mathcal{B}$
and $L,R:\mathcal{B}\to \mathcal{A}$ between abelian categories. 
\begin{enumerate}
\item Let $M$ and $N$ be objects of $\mathcal{A}$, and assume that $F$ is fully faithful. Then $N$ is (dual)
$M$-Rickart in $\mathcal{A}$ if and only if $F(N)$ is (dual) $F(M)$-Rickart in $\mathcal{B}$.
\item Let $M$ and $N$ be objects of $\mathcal{B}$, and assume that $L$ (or $R$) is fully faithful. Then: 
\begin{enumerate}[(i)] 
\item $N$ is $M$-Rickart in $\mathcal{B}$ if and only if $R(N)$ is $R(M)$-Rickart in $\mathcal{A}$.
\item $N$ is dual $M$-Rickart in $\mathcal{B}$ if and only if $L(N)$ is dual $L(M)$-Rickart in $\mathcal{A}$.
\end{enumerate}
\end{enumerate}
\end{coll}

Let $\varphi:R\to S$ be a ring homomorphism. Following \cite[Chapter~IX, p.105]{St}, consider the following covariant
functors: extension of scalars $\varphi^*:{\rm Mod}(R)\to {\rm Mod}(S)$ given on objects by $\varphi^*(M)=M\otimes_RS$,
restriction of scalars $\varphi_*:{\rm Mod}(S)\to {\rm Mod}(R)$ given on objects by $\varphi_*(N)=N$, and
$\varphi^!:{\rm Mod}(R)\to {\rm Mod}(S)$ given on objects by $\varphi^!(M)={\rm Hom}_R(S,M)$. Then
$(\varphi^*,\varphi_*,\varphi^!)$ is an adjoint triple of functors. 

\begin{coll} Let $\varphi:R\to S$ be a ring epimorphism, and let $M$ and $N$ be right $S$-modules. Then $N$ is a (dual)
$M$-Rickart right $S$-module if and only if $N$ is a (dual) $M$-Rickart right $R$-module.
\end{coll}

\begin{proof} Since $\varphi:R\to S$ is a ring epimorphism, the restriction of scalars functor $\varphi_*:{\rm
Mod}(S)\to {\rm Mod}(R)$ is fully faithful \cite[Chapter~XI, Proposition~1.2]{St}. Then use Corollary \ref{c:tripleff}
for the adjoint triple of functors $(\varphi^*,\varphi_*,\varphi^!)$.
\end{proof}

Let $\mathcal{A}$ be an abelian category and let $\mathcal{C}$ be a full subcategory of $\mathcal{A}$. Then
$\mathcal{C}$ is called a \emph{reflective} (\emph{coreflective}) subcategory of $\mathcal{A}$ if the inclusion functor
$i:\mathcal{C}\to \mathcal{A}$ has a left (right) adjoint. In this case $i$ is fully faithful. Then Theorem \ref{t:ff} yields the following corollary. 

\begin{coll} \label{c:rc} Let $\mathcal{A}$ be an abelian category, $\mathcal{C}$ an abelian full subcategory of
$\mathcal{A}$ and $i:\mathcal{C}\to \mathcal{A}$ the inclusion functor. Let $M$ and $N$ be objects of $\mathcal{C}$. 
\begin{enumerate}
\item Assume that $\mathcal{C}$ is a reflective subcategory of $\mathcal{A}$. 
Then $N$ is $M$-Rickart in $\mathcal{C}$ if and only if $i(N)$ is $i(M)$-Rickart in $\mathcal{A}$.
\item Assume that $\mathcal{C}$ is a coreflective subcategory of $\mathcal{A}$. 
Then $N$ is (dual) $M$-Rickart in $\mathcal{C}$ if and only if $i(N)$ is (dual) $i(M)$-Rickart in $\mathcal{A}$.
\end{enumerate}
\end{coll}

Following \cite[Section~2.2]{DNR}, let $C$ be a coalgebra over a field $k$, and let ${}^C\mathcal{M}$ be the
(Grothendieck) category of left $C$-comodules. Left $C$-comodules may be and will be identified with rational right
$C^*$-comodules, where $C^*={\rm Hom}_{k}(C,k)$. Consider the inclusion functor $i:{}^C\mathcal{M}\to {\rm Mod}(C^*)$,
and the functor ${\rm Rat}:{\rm Mod}(C^*)\to {}^C\mathcal{M}$ which associates to every right $C^*$-module its rational
$C^*$-submodule. Then $i$ is a fully faithful exact functor and ${\rm Rat}$ is a right adjoint to $i$. Hence
${}^C\mathcal{M}$ is a coreflective subcategory of ${\rm Mod}(C^*)$. 

\begin{coll} \label{c:com1} Let $C$ be a coalgebra over a field, and let $M$ and $N$ be left $C$-comodules. Then:
\begin{enumerate} 
\item $N$ is $M$-Rickart if and only if $N$ is $M$-Rickart as a right $C^*$-module. 
\item $N$ is dual $M$-Rickart if and only if $N$ is dual $M$-Rickart as a right $C^*$-module. 
\end{enumerate}
\end{coll}

\begin{proof} Note that $i:{}^C\mathcal{M}\to {\rm Mod}(C^*)$ is a fully faithful exact functor, and use Theorem \ref{t:ff}.
\end{proof}

Theorem \ref{t:ff} may be partially generalized from left (right) exact fully faithful functors to Maschke functors.
Recall that a functor $F:\mathcal{A}\to \mathcal{B}$ between abelian categories is called a \emph{Maschke functor} if
$F$ is left exact and for any object $M$ in $\mathcal{A}$ and any subobject $K$ of $M$ such that $F(K)$ is a direct
summand of $F(M)$, we have that $K$ is a direct summand of $M$. Dually, one defines the notion of \emph{dual Maschke
functor}. It is easy to check that a left (right) exact fully faithful functor is a Maschke (dual Maschke) functor. 

\begin{theo} \label{t:maschke} Let $F:\mathcal{A}\to \mathcal{B}$ be a functor between abelian categories, and
let $M$ and $N$ be objects of $\mathcal{A}$.
\begin{enumerate}
\item If $F$ is Maschke and $F(N)$ is $F(M)$-Rickart in $\mathcal{B}$, then $N$ is $M$-Rickart in $\mathcal{A}$.
\item If $F$ is dual Maschke and $F(N)$ is dual $F(M)$-Rickart in $\mathcal{B}$, then $N$ is dual $M$-Rickart in
$\mathcal{A}$.
\end{enumerate}
\end{theo}

\begin{proof} (1) Let $f:M\to N$ be a morphism in $\mathcal{A}$ and consider the morphism $F(f):F(M)\to F(N)$ in
$\mathcal{B}$. Since $F(N)$ is $F(M)$-Rickart, ${\rm Ker}(F(f))$ is a direct summand of $F(M)$. Since $F$ is a
Maschke functor, it is left exact, hence we have ${\rm ker}(F(f))=F({\rm ker}(f))$. Moreover, it follows that
${\rm Ker}(f)$ is a direct summand of $M$. Hence $N$ is $M$-Rickart. 
\end{proof}

Let $(L,R)$ be an adjoint pair of covariant functors $L:\mathcal{A}\to \mathcal{B}$ and
$R:\mathcal{B}\to \mathcal{A}$ between abelian categories such that $R$ is fully faithful. 
Also, let $M$ and $N$ be objects of $\mathcal{B}$. Since $R$ is left exact, $N$ is $M$-Rickart in $\mathcal{B}$ 
if and only if $R(N)$ is $R(M)$-Rickart in $\mathcal{A}$ by Theorem \ref{t:ff}.
Since $R$ is fully faithful, one has $LR\cong 1_{\mathcal{B}}$, and in particular, $LR(M)\cong M$ and $LR(N)\cong N$. 
A version of transfer of the relative Rickart property will still hold if the global hypothesis of 
a fully faithful functor $R$ is replaced by a local one, in the sense that we only ask 
the isomorphisms $LR(M)\cong M$ and $LR(N)\cong N$. But we need to add some extra conditions, 
using the following concepts. 

\begin{defn} \rm Let $M$ and $N$ be objects of an abelian category $\mathcal{A}$. Then: 
\begin{enumerate}
\item $N$ is called \emph{$M$-cyclic} if there exists an epimorphism $M\to N$.
\item $M$ is called \emph{$N$-cocyclic} if there exists a monomorphism $M\to N$.
\end{enumerate}
\end{defn}

Let $(L,R)$ be an adjoint pair of covariant functors $L:\mathcal{A}\to \mathcal{B}$ and $R:\mathcal{B}\to \mathcal{A}$
between abelian categories. Let $\varepsilon:LR\to 1_{\mathcal{B}}$ and $\eta:1_{\mathcal{A}}\to RL$ be the counit and
the unit of adjunction respectively. Recall that an object $B\in \mathcal{B}$ is called \emph{$R$-static} if
$\varepsilon_B$ is an isomorphism, while an object $A\in \mathcal{A}$ is called \emph{$R$-adstatic} if $\eta_A$ is an
isomorphism \cite{CGW}. Denote by ${\rm Stat}(R)$ the full subcategory of $\mathcal{B}$ consisting of $R$-static
objects, and by ${\rm Adst}(R)$ the full subcategory of $\mathcal{A}$ consisting of $R$-adstatic objects. 

\begin{theo} \label{t:equiv} Let $(L,R)$ be an adjoint pair of covariant functors $L:\mathcal{A}\to \mathcal{B}$ and
$R:\mathcal{B}\to \mathcal{A}$ between abelian categories. 
\begin{enumerate}
\item Let $M$ and $N$ be objects of $\mathcal{B}$ such that $M,N\in {\rm Stat}(R)$. Then the following are equivalent:
\begin{enumerate}[(i)] 
\item $N$ is $M$-Rickart in $\mathcal{B}$.
\item $R(N)$ is $R(M)$-Rickart in $\mathcal{A}$ and for every morphism $f:M\to N$, ${\rm Ker}(f)$ is $M$-cyclic.
\item $R(N)$ is $R(M)$-Rickart in $\mathcal{A}$ and for every morphism $f:M\to N$, ${\rm Ker}(f)\in {\rm Stat}(R)$.
\end{enumerate}
\item Let $M$ and $N$ be objects of $\mathcal{A}$ such that $M,N\in {\rm Adst}(R)$. Then the following are equivalent:
\begin{enumerate}[(i)] 
\item $N$ is dual $M$-Rickart in $\mathcal{A}$.
\item $L(N)$ is dual $L(M)$-Rickart in $\mathcal{B}$ and for every morphism $f:M\to N$, ${\rm Coker}(f)$ is
$N$-cocyclic.
\item $L(N)$ is dual $L(M)$-Rickart in $\mathcal{B}$ and for every morphism $f:M\to N$, ${\rm Coker}(f)\in {\rm
Adst}(R)$.
\end{enumerate}
\end{enumerate}
\end{theo}

\begin{proof} Let $\varepsilon:LR\to 1_{\mathcal{B}}$ and $\eta:1_{\mathcal{A}}\to RL$ be the counit and the unit of
adjunction respectively.

(1) (i)$\Rightarrow$(ii) Assume that $N$ is $M$-Rickart in $\mathcal{B}$. Let $g:R(M)\to R(N)$ be a morphism in
$\mathcal{A}$. By naturality we have the following commutative diagram in $\mathcal{A}$: 
\[\SelectTips{cm}{}
\xymatrix{
R(M) \ar[d]_{\eta_{R(M)}} \ar[r]^g & R(N) \ar[d]^{\eta_{R(N)}} \\ 
RLR(M) \ar[r]_{RL(g)} & RLR(N)  
} 
\]
Since $R(\varepsilon_M)\eta_{R(M)}=1_{R(M)}$ and $\varepsilon_M$ is an isomorphism, it follows that $\eta_{R(M)}$ is an
isomorphism. Similarly, $\eta_{R(N)}$ is an isomorphism. Then we have $RL(g)=g$. Viewing $L(g):M\to N$, ${\rm
ker}(L(g))$ is a section, because $N$ is $M$-Rickart. But $R$ is left exact, hence  ${\rm ker}(g)={\rm
ker}(RL(g))=R({\rm ker}(L(g)))$ is a section. Thus $R(N)$ is $R(M)$-Rickart in $\mathcal{A}$. Finally, for every
morphism $f:M\to N$ in $\mathcal{B}$, ${\rm Ker}(f)$ is a direct summand of $M$, hence ${\rm Ker}(f)$ is $M$-cyclic.

(ii)$\Rightarrow$(iii) Assume that $R(N)$ is $R(M)$-Rickart in $\mathcal{A}$ and for every morphism $f:M\to
N$, ${\rm Ker}(f)$ is $M$-cyclic. Let $f:M\to N$ be a morphism in $\mathcal{B}$ with kernel $k={\rm
ker}(f):K\to M$. Since $R(N)$ is $R(M)$-Rickart, ${\rm ker}(R(f))$ is a section. Since $R$ is left
exact, it follows that $R(k)$ is a section, and so $LR(k)$ is a section. Since $K$ is $M$-cyclic, there is an
epimorphism $p:M\to K$. By naturality we have the following two commutative diagrams in $\mathcal{B}$: 
\[\SelectTips{cm}{}
\xymatrix{
LR(K) \ar[d]_{\varepsilon_K} \ar[r]^{LR(k)} & LR(M) \ar[d]^{\varepsilon_M} \\ 
K \ar[r]_k & M  
} 
\hspace{2cm}
\xymatrix{
LR(M) \ar[d]_{\varepsilon_M} \ar[r]^{LR(p)} & LR(K) \ar[d]^{\varepsilon_K} \\ 
M \ar[r]_p & K  
} 
\]
Then $k\varepsilon_K=\varepsilon_M LR(k)$ is a section, hence $\varepsilon_K$ is a section. Also, $\varepsilon_K
LR(p)=p\varepsilon_M$ is an epimorphism, hence $\varepsilon_K$ is an epimorphism. Then $\varepsilon_K$ is
an isomorphism, which implies that ${\rm Ker}(f)\in {\rm Stat}(R)$. 

(iii)$\Rightarrow$(i) Assume that $R(N)$ is $R(M)$-Rickart in $\mathcal{A}$ and for every morphism $f:M\to
N$, ${\rm Ker}(f)\in {\rm Stat}(R)$. Let $f:M\to N$ be a morphism in $\mathcal{B}$ with kernel $k={\rm
ker}(f):K\to M$. Then $\varepsilon_K$ and $\varepsilon_M$ are isomorphisms, and $LR(k)$ is a section as above. Now
the above left hand side diagram implies that $k$ is a section. Hence $N$ is $M$-Rickart in $\mathcal{B}$.
\end{proof}

Theorem \ref{t:equiv} together with its version for the contravariant case yield our main application to endomorphism
rings of modules in Section 4. That is why we give the necessary associated concepts, and we state it for easy reference.

Let $(L,R)$ be a right adjoint pair of contravariant functors $L:\mathcal{A}\to
\mathcal{B}$ and $R:\mathcal{B}\to \mathcal{A}$ between abelian categories, in the sense of \cite[p.~81]{Freyd}. 
Let $\varepsilon:1_{\mathcal{B}}\to LR$ and $\eta:1_{\mathcal{A}}\to RL$ be the counit and the unit of adjunction respectively. 
Recall that an object $B\in \mathcal{B}$ is called \emph{$R$-reflexive} if $\varepsilon_B$ is an isomorphism, 
while an object $A\in \mathcal{A}$ is called \emph{$L$-reflexive} if $\eta_A$ is an isomorphism \cite{Castano}. 
Denote by ${\rm Refl}(R)$ the full subcategory of $\mathcal{B}$ consisting of $R$-reflexive objects, 
and by ${\rm Refl}(L)$ the full subcategory of $\mathcal{A}$ consisting of $L$-reflexive objects. 
One can consider similar notations for a left adjoint pair of contravariant functors.  

\begin{theo} \label{t:dual} Let $(L,R)$ be a pair of contravariant functors $L:\mathcal{A}\to \mathcal{B}$ and
$R:\mathcal{B}\to \mathcal{A}$ between abelian categories. 
\begin{enumerate}
\item Assume that $(L,R)$ is left adjoint. Let $M$ and $N$ be objects of $\mathcal{B}$ such that $M,N\in {\rm Refl}(R)$.
Then the following are equivalent:
\begin{enumerate}[(i)] 
\item $N$ is $M$-Rickart in $\mathcal{B}$.
\item $R(M)$ is dual $R(N)$-Rickart in $\mathcal{A}$ and for every morphism $f:M\to N$, ${\rm Ker}(f)$ is $M$-cyclic.
\item $R(M)$ is dual $R(N)$-Rickart in $\mathcal{A}$ and for every morphism $f:M\to N$, ${\rm Ker}(f)\in {\rm Refl}(R)$.
\end{enumerate}
\item Assume that $(L,R)$ is right adjoint. Let $M$ and $N$ be objects of $\mathcal{A}$ such that $M,N\in {\rm
Refl}(L)$. Then the following are equivalent:
\begin{enumerate}[(i)] 
\item $N$ is dual $M$-Rickart in $\mathcal{A}$.
\item $L(M)$ is $L(N)$-Rickart in $\mathcal{B}$ and for every morphism $f:M\to N$, ${\rm Coker}(f)$ is $N$-cocyclic.
\item $L(M)$ is $L(N)$-Rickart in $\mathcal{B}$ and for every morphism $f:M\to N$, ${\rm Coker}(f)\in {\rm Refl}(L)$.
\end{enumerate}
\end{enumerate}
\end{theo}

Following \cite{C08}, for a right $R$-module $P$ we denote by ${\rm PAdd}(P)$ the class of right $R$-modules $Z$ for
which there is a pure epimorphism $P^{(I)}\to Z$ for some set $I$. 
Note that the full subcategory ${\rm PAdd}(P)$ of ${\rm Mod}(R)$ and the full subcategory of flat right ${\rm End}_R(P)$-modules
are not abelian in general \cite[Theorem~3]{GM}. 

\begin{coll} \label{c:PAdd} Let $P$ be a finitely presented right $R$-module, and let $S={\rm End}_R(P)$. 
\begin{enumerate}
\item Let $M,N\in {\rm PAdd}(P)$. Then the following are equivalent:
\begin{enumerate}[(i)] 
\item $N$ is $M$-Rickart.
\item ${\rm Hom}_R(P,N)$ is a ${\rm Hom}_R(P,M)$-Rickart right $S$-module and for every homomorphism $f:M\to N$, ${\rm
Ker}(f)$ is $M$-cyclic.
\item ${\rm Hom}_R(P,N)$ is a ${\rm Hom}_R(P,M)$-Rickart right $S$-module and for every homomorphism $f:M\to N$, ${\rm
Ker}(f)\in {\rm Stat}({\rm Hom}_R(P,-))$.
\end{enumerate}
\item Let $M'$ and $N'$ be flat right $S$-modules. Then the following are equivalent:
\begin{enumerate}[(i)] 
\item $N'$ is dual $M'$-Rickart.
\item $N'\otimes_SP$ is a dual $M'\otimes_SP$-Rickart right $R$-module and for every homomorphism $f:M'\to N'$, ${\rm
Coker}(f)$ is $N'$-cocyclic.
\item $N'\otimes_SP$ is a dual $M'\otimes_SP$-Rickart right $R$-module and for every homomorphism $f:M'\to N'$, ${\rm
Coker}(f)\in {\rm Adst}({\rm Hom}_R(P,-))$.
\end{enumerate}
\end{enumerate}
\end{coll}

\begin{proof} Consider the adjoint pair of covariant functors $(T,H)$, where $T=-\otimes_SP:{\rm Mod}(S)\to {\rm
Mod}(R)$ and $H={\rm Hom}_R(P,-):{\rm Mod}(R)\to {\rm Mod}(S)$. Then the pair $(T,H)$ induces an equivalence between
the full subcategory ${\rm PAdd}(P)$ of ${\rm Mod}(R)$ and the full subcategory of flat right $S$-modules by
\cite[Lemma~2.4]{GG}. Now the conclusion follows by Theorem \ref{t:equiv}.  
\end{proof}

\section{(Dual) relative Baer objects}

Let us recall the definitions of (dual) relative Baer objects and their connections to (dual) relative Rickart objects
in abelian categories.

\begin{defn} \cite[Definition~6.1]{CK} \rm Let $M$ and $N$ be objects of an abelian category $\mathcal{A}$. Then $N$ is
called: 
\begin{enumerate}
\item \emph{$M$-Baer} if for every family $(f_i)_{i\in I}$ with each $f_i\in \Hom_{\mathcal{A}}(M,N)$, $\bigcap_{i\in I}
{\rm Ker}(f_i)$ is a direct summand of $M$. 
\item \emph{dual $M$-Baer} if for every family $(f_i)_{i\in I}$ with each $f_i\in \Hom_{\mathcal{A}}(M,N)$, $\sum_{i\in
I} {\rm Im}(f_i)$ is a direct summand of $N$.  
\item \emph{self-Baer} if $N$ is $N$-Baer.
\item \emph{dual self-Baer} if $N$ is dual $N$-Baer.
\end{enumerate} 
\end{defn}

\begin{rem} \rm Our (dual) self-Baer objects are simply called (dual) Baer modules in module categories \cite{KT,RR04}.
\end{rem}

\begin{lemm} \cite[Lemma~6.2]{CK} \label{l:BR} Let $M$ and $N$ be objects of an abelian category $\mathcal{A}$.
\begin{enumerate}
\item Assume that there exists the product $N^I$ for every set $I$. Then $N$ is $M$-Baer if and only if $N^I$ is
$M$-Rickart for every set $I$.
\item Assume that there exists the coproduct $M^{(I)}$ for every set $I$. Then $N$ is dual $M$-Baer if and only
if $N$ is dual $M^{(I)}$-Rickart
for every set $I$. 
\end{enumerate}
\end{lemm}

\begin{coll} \label{c:ffbaer} Let $F:\mathcal{A}\to \mathcal{B}$ be a fully faithful functor between abelian categories,
and let $M$ and $N$ be objects of $\mathcal{A}$. 
\begin{enumerate}
\item Assume that there exists the product $N^I$ for every set $I$, $F$ is left exact and $F$ preserves products. Then
$N$ is $M$-Baer in $\mathcal{A}$ if and only if $F(N)$ is $F(M)$-Baer in $\mathcal{B}$.
\item Assume that there exists the coproduct $M^{(I)}$ for every set $I$, $F$ is right exact and $F$ preserves
coproducts. Then $N$ is dual $M$-Baer in $\mathcal{A}$ if and only if $F(N)$ is dual $F(M)$-Baer in $\mathcal{B}$.
\end{enumerate}
\end{coll}

\begin{proof} This follows by Theorem \ref{t:ff} and Lemma \ref{l:BR}.
\end{proof}

For Grothendieck categories we have the following corollary.

\begin{coll} Let $\mathcal{A}$ be a Grothendieck category with a generator $U$, $R={\rm End}_{\mathcal{A}}(U)$, $S={\rm
Hom}_{\mathcal{A}}(U,-):\mathcal{A}\to {\rm Mod}(R)$. Let $M$ and $N$ be objects of $\mathcal{A}$. Then $N$ is an
$M$-Baer object of $\mathcal{A}$ if and only if $S(N)$ is an $S(M)$-Baer right $R$-module.
\end{coll}

\begin{proof} By the Gabriel-Popescu Theorem \cite[Chapter~X, Theorem~4.1]{St}, $S$ is a fully faithful functor which
has an exact left adjoint $T:{\rm Mod}(R)\to \mathcal{A}$. Since $S$ is a right adjoint, it is left exact and preserves
products. Then use Corollary \ref{c:ffbaer}. 
\end{proof}

\begin{coll} \label{c:tripleffbaer} Let $(L,F,R)$ be an adjoint triple of covariant functors $F:\mathcal{A}\to
\mathcal{B}$ and $L,R:\mathcal{B}\to \mathcal{A}$ between abelian categories. 
\begin{enumerate}
\item Let $M$ and $N$ be objects of $\mathcal{A}$, and assume that $F$ is fully faithful. 
\begin{enumerate}[(i)] 
\item Assume that there exists the product $N^I$ for every set $I$. Then $N$ is $M$-Baer in $\mathcal{A}$ if and only if
$F(N)$ is $F(M)$-Baer in $\mathcal{B}$.
\item Assume that there exists the coproduct $M^{(I)}$ for every set $I$. Then $N$ is dual $M$-Baer in $\mathcal{A}$ if
and only if $F(N)$ is dual $F(M)$-Baer in $\mathcal{B}$.
\end{enumerate}
\item Let $M$ and $N$ be objects of $\mathcal{B}$, and assume that $L$ (or $R$) is fully faithful. 
\begin{enumerate}[(i)] 
\item Assume that there exists the product $N^I$ for every set $I$. Then $N$ is $M$-Baer in $\mathcal{B}$ if and only if
$R(N)$ is $R(M)$-Baer in $\mathcal{A}$.
\item Assume that there exists the coproduct $M^{(I)}$ for every set $I$. Then $N$ is dual $M$-Baer in $\mathcal{B}$ if
and only if $L(N)$ is dual $L(M)$-Baer in $\mathcal{A}$.
\end{enumerate}
\end{enumerate}
\end{coll}

\begin{proof} This follows by Corollary \ref{c:tripleff}, Lemma \ref{l:BR} and the facts that $F$ preserves products
and coproducts as a left and right adjoint, $R$ preserves products, and $L$ preserves coproducts. 
\end{proof}

\begin{coll} Let $\varphi:R\to S$ be a ring epimorphism, and let $M$ and $N$ be right $S$-modules. Then $N$ is a (dual)
$M$-Baer right $S$-module if and only if $N$ is a (dual) $M$-Baer right $R$-module.
\end{coll}

\begin{proof} Since $\varphi:R\to S$ is a ring epimorphism, the restriction of scalars functor $\varphi_*:{\rm
Mod}(S)\to {\rm Mod}(R)$ is fully faithful \cite[Chapter~XI, Proposition~1.2]{St}. Then use Corollary
\ref{c:tripleffbaer} for the adjoint triple of functors $(\varphi^*,\varphi_*,\varphi^!)$.
\end{proof}

\begin{coll} \label{c:rcbaer} Let $\mathcal{A}$ be an abelian category, $\mathcal{C}$ an abelian full subcategory of
$\mathcal{A}$ and $i:\mathcal{C}\to \mathcal{A}$ the inclusion functor. Let $M$ and $N$ be objects of $\mathcal{C}$. 
\begin{enumerate}
\item Assume that $\mathcal{C}$ is a reflective subcategory of $\mathcal{A}$, 
and there exists the product $N^I$ for every set $I$. Then $N$ is $M$-Baer in
$\mathcal{C}$ if and only if $i(N)$ is $i(M)$-Baer in $\mathcal{A}$.
\item Assume that $\mathcal{C}$ is a coreflective subcategory of $\mathcal{A}$, 
and there exists the coproduct $M^{(I)}$ for every set $I$. Then $N$ is (dual)
$M$-Baer in $\mathcal{C}$ if and only if $i(N)$ is (dual) $i(M)$-Baer in $\mathcal{A}$.
\end{enumerate}
\end{coll}

\begin{proof} Note that $i$ is fully faithful left exact and preserves products in (1), 
and $i$ is fully faithful exact and preserves coproducts in (2). Then the conclusion follows by Corollary \ref{c:ffbaer}.
\end{proof}

For comodule categories we have the following corollary.

\begin{coll} \label{c:com3} Let $C$ be a coalgebra over a field, and let $M$ and $N$ be left $C$-comodules. Then:
\begin{enumerate} 
\item $N$ is $M$-Baer if and only if $N$ is $M$-Baer as a right $C^*$-module. 
\item $N$ is dual $M$-Baer if and only if $N$ is dual $M$-Baer as a right $C^*$-module. 
\end{enumerate}
\end{coll}

\begin{proof} Note that $i:{}^C\mathcal{M}\to {\rm Mod}(C^*)$ is a fully faithful exact functor, and use Corollary \ref{c:ffbaer}.
\end{proof}

\begin{coll} \label{c:maschkebaer} Let $F:\mathcal{A}\to \mathcal{B}$ be a Maschke functor between abelian categories,
and let $M$ and $N$ be objects of $\mathcal{A}$.
\begin{enumerate}
\item Assume that there exists the product $N^I$ for every set $I$, and $F$ preserves products. If
$F(N)$ is $F(M)$-Baer in $\mathcal{B}$, then $N$ is $M$-Baer in $\mathcal{A}$.
\item Assume that there exists the coproduct $M^{(I)}$ for every set $I$, and $F$ preserves coproducts. If $F(N)$ is
dual $F(M)$-Baer in $\mathcal{B}$, then $N$ is dual $M$-Baer in $\mathcal{A}$.
\end{enumerate}
\end{coll}

\begin{proof} This follows by Theorem \ref{t:maschke} and Lemma \ref{l:BR}.
\end{proof}

The following easy general lemma will be useful.

\begin{lemm} \label{l:kerim} Let $M$ and $N$ be objects of an abelian category $\mathcal{A}$, and let $\mathcal{P}$ be
a property for objects of $\mathcal{A}$.
\begin{enumerate}
\item Assume that there exists the product $N^I$ for every set $I$. Then the following are equivalent:
\begin{enumerate}[(i)]
\item For every morphism $f:M\to N^I$ for some set $I$, ${\rm Ker}(f)$ has $\mathcal{P}$.
\item For every family $(f_i)_{i\in I}$ with each $f_i\in {\rm Hom}_{\mathcal{A}}(M,N)$, $\bigcap_{i\in I}{\rm
Ker}(f_i)$ has $\mathcal{P}$.
\end{enumerate}
\item Assume that there exists the coproduct $M^{(I)}$ for every set $I$. Then the following are equivalent:
\begin{enumerate}[(i)]
\item For every morphism $f:M^{(I)}\to N$ for some set $I$, ${\rm Coker}(f)$ has $\mathcal{P}$.
\item For every family $(f_i)_{i\in I}$ with each $f_i\in {\rm Hom}_{\mathcal{A}}(M,N)$, $\sum_{i\in I} {\rm Im}(f_i)$
has $\mathcal{P}$.
\end{enumerate}
\end{enumerate}
\end{lemm}

\begin{proof} (1) Let $I$ be a set. For every $i\in I$, denote by $p_i:N^I\to N$ the canonical projection.

(i)$\Rightarrow$(ii) Assume that (i) holds. Let $(f_i)_{i\in I}$ be a family with each $f_i\in
\Hom_{\mathcal{A}}(M,N)$. By the universal property of the product, there exists a morphism $f:M\to N^I$ such that
$p_if=f_i$ for every $i\in I$. Then ${\rm ker}(f)=\bigcap_{i\in I} {\rm ker}(f_i)$ has $\mathcal{P}$ by hypothesis. 

(ii)$\Rightarrow$(i) Assume that (ii) holds. Let $f:M\to N^I$ be a morphism in $\mathcal{A}$ for some set $I$. For every
$i\in I$, denote $f_i=p_if:M\to N$. Then $\bigcap_{i\in I} {\rm ker}(f_i)={\rm ker}(f)$ has $\mathcal{P}$ by
hypothesis. 
\end{proof}

\begin{coll} \label{c:equivbaer} Let $(L,R)$ be an adjoint pair of covariant functors $L:\mathcal{A}\to \mathcal{B}$ and
$R:\mathcal{B}\to \mathcal{A}$ between abelian categories. 
\begin{enumerate}
\item Let $M$ and $N$ be objects of $\mathcal{B}$ such that $M,N\in {\rm Stat}(R)$ and for every set $I$ there exists
the product $N^I$. Then the following are equivalent:
\begin{enumerate}[(i)]
\item $N$ is $M$-Baer in $\mathcal{B}$.
\item $R(N)$ is $R(M)$-Baer in $\mathcal{A}$ and for every family $(f_i)_{i\in I}$ with each $f_i\in {\rm
Hom}_{\mathcal{A}}(M,N)$, $\bigcap_{i\in I}{\rm Ker}(f_i)$ is $M$-cyclic.
\item $R(N)$ is $R(M)$-Baer in $\mathcal{A}$ and for every family $(f_i)_{i\in I}$ with each $f_i\in {\rm
Hom}_{\mathcal{A}}(M,N)$, $\bigcap_{i\in I}{\rm Ker}(f_i)\in {\rm Stat}(R)$.
\end{enumerate}
\item Let $M$ and $N$ be objects of $\mathcal{A}$ such that $M,N\in {\rm Adst}(R)$ and for every set $I$ there exists
the coproduct $M^{(I)}$. Then the following are equivalent:
\begin{enumerate}[(i)]
\item $N$ is dual $M$-Baer in $\mathcal{A}$.
\item $L(N)$ is dual $L(M)$-Baer in $\mathcal{B}$ and for every family $(f_i)_{i\in I}$ with each $f_i\in {\rm
Hom}_{\mathcal{A}}(M,N)$, $\sum_{i\in I} {\rm Im}(f_i)$ is $N$-cocyclic.
\item $L(N)$ is dual $L(M)$-Baer in $\mathcal{B}$ and for every family $(f_i)_{i\in I}$ with each $f_i\in {\rm
Hom}_{\mathcal{A}}(M,N)$, $\sum_{i\in I} {\rm Im}(f_i)\in {\rm Adst}(R)$.
\end{enumerate}
\end{enumerate}
\end{coll}

\begin{proof} This follows by Theorem \ref{t:equiv}, Lemmas \ref{l:BR} and \ref{l:kerim} and the facts that $R$
preserves products and $L$ preserves coproducts.
\end{proof}

\begin{coll} Let $(L,R)$ be a pair of contravariant functors $L:\mathcal{A}\to \mathcal{B}$ and
$R:\mathcal{B}\to \mathcal{A}$ between abelian categories. 
\begin{enumerate}
\item Assume that $(L,R)$ is left adjoint. Let $M$ and $N$ be objects of $\mathcal{B}$ such that $M,N\in {\rm Refl}(R)$
and for every set $I$ there exists the product $N^I$. Then the following are equivalent:
\begin{enumerate}[(i)]
\item $N$ is $M$-Baer in $\mathcal{B}$.
\item $R(M)$ is dual $R(N)$-Baer in $\mathcal{A}$ and for every set $I$ and for every family $(f_i)_{i\in
I}$ with each $f_i\in {\rm Hom}_{\mathcal{A}}(M,N)$, $\bigcap_{i\in I}{\rm Ker}(f_i)$ is $M$-cyclic.
\item $R(M)$ is dual $R(N)$-Baer in $\mathcal{A}$ and for every set $I$ and for every family $(f_i)_{i\in
I}$ with each $f_i\in {\rm Hom}_{\mathcal{A}}(M,N)$, $\bigcap_{i\in I}{\rm Ker}(f_i)\in {\rm Refl}(R)$.
\end{enumerate}
\item Assume that $(L,R)$ is right adjoint. Let $M$ and $N$ be objects of $\mathcal{A}$ such that $M,N\in {\rm Refl}(L)$
and for every set $I$ there exists the coproduct $M^{(I)}$. Then the following are equivalent:
\begin{enumerate}[(i)]
\item $N$ is dual $M$-Baer in $\mathcal{A}$.
\item $L(M)$ is $L(N)$-Baer in $\mathcal{B}$ and for every set $I$ and for every family $(f_i)_{i\in I}$ with each
$f_i\in {\rm Hom}_{\mathcal{A}}(M,N)$, $\sum_{i\in I} {\rm Im}(f_i)$ is $N$-cocyclic.
\item $L(M)$ is $L(N)$-Baer in $\mathcal{B}$ and for every set $I$ and for every family $(f_i)_{i\in I}$ with each
$f_i\in {\rm Hom}_{\mathcal{A}}(M,N)$, $\sum_{i\in I} {\rm Im}(f_i)\in {\rm Refl}(L)$.
\end{enumerate}
\end{enumerate}
\end{coll}

\begin{proof} (2) This follows by Theorem \ref{t:dual}, Lemmas \ref{l:BR} and \ref{l:kerim} and the fact that $L$ converts
coproducts into products.
\end{proof}

\begin{coll} Let $P$ be a finitely presented right $R$-module, and let $S={\rm End}_R(P)$. 
\begin{enumerate}
\item Let $M$ and $N$ be right $R$-modules such that $M,N\in {\rm PAdd}(P)$. The following are equivalent:
\begin{enumerate}[(i)] 
\item $N$ is $M$-Baer.
\item ${\rm Hom}_R(P,N)$ is a ${\rm Hom}_R(P,M)$-Baer right $S$-module for every family $(f_i)_{i\in I}$ with
each $f_i\in {\rm Hom}_{\mathcal{A}}(M,N)$, $\bigcap_{i\in I}{\rm Ker}(f_i)$ is $M$-cyclic.
\item ${\rm Hom}_R(P,N)$ is a ${\rm Hom}_R(P,M)$-Baer right $S$-module and for every family $(f_i)_{i\in I}$ with
each $f_i\in {\rm Hom}_{\mathcal{A}}(M,N)$, $\bigcap_{i\in I}{\rm Ker}(f_i)\in {\rm Stat}({\rm Hom}_R(P,-))$.
\end{enumerate}
\item Let $M'$ and $N'$ be right $S$-modules such that $M',N'$ are flat. The following are equivalent:
\begin{enumerate}[(i)] 
\item $N'$ is dual $M'$-Baer.
\item $N'\otimes_SP$ is a dual $M'\otimes_SP$-Baer right $R$-module and for every family $(f_i)_{i\in I}$ with each
$f_i\in {\rm Hom}_{\mathcal{A}}(M',N')$, $\sum_{i\in I} {\rm Im}(f_i)$ is $N'$-cocyclic.
\item $N'\otimes_SP$ is a dual $M'\otimes_SP$-Baer right $R$-module and for every family $(f_i)_{i\in I}$ with each
$f_i\in {\rm Hom}_{\mathcal{A}}(M',N')$, $\sum_{i\in I} {\rm Im}(f_i)\in {\rm Adst}({\rm Hom}_R(P,-))$.
\end{enumerate}
\end{enumerate}
\end{coll}

\begin{proof} See the proof of Corollary \ref{c:PAdd}, and use Corollary \ref{c:equivbaer}.
\end{proof}

\section{Endomorphism rings}

In this section we discuss the transfer of (dual) relative Rickart and (dual) relative Baer properties to endomorphism
rings of (graded) modules.

\subsection{Modules} Let us recall some terminology and notation in module categories. We need the concepts of locally
split monomorphism and locally split epimorphism due to Azumaya. Recall that a monomorphism $f:A\to B$ of right
$R$-modules is called \emph{locally split} if for every $a\in A$, there exists an $R$-homomorphism $h:B\to A$ such that
$h(f(a))=a$, while an epimorphism $g:B\to C$ of right $R$-modules is called \emph{locally split} if for every $c\in C$,
there exists an $R$-homomorphism $h:C\to B$ such that $g(h(c))=c$ \cite[p.~132]{Azumaya}. 

Recall that a right $R$-module $M$ is called \emph{quasi-retractable} if for every family $(f_i)_{i\in I}$ with each
$f_i\in {\rm End}_R(M)$ and $\bigcap_{i\in I} {\rm Ker}(f_i)\neq 0$, ${\rm Hom}_R(M,\bigcap_{i\in I} {\rm Ker}(f_i))\neq
0$ \cite[Definition~2.3]{RR09}. Dually, a right $R$-module is called \emph{quasi-coretractable} if for every family
$(f_i)_{i\in I}$ with each $f_i\in {\rm End}_R(M)$ and $\sum_{i\in I} {\rm Im}(f_i)\neq M$, ${\rm Hom}_R(M/\sum_{i\in I}
{\rm Im}(f_i),M)\neq 0$ \cite[Definition~3.2]{KST}. They are useful when relating (dual) self-Baer properties of a
module and of its endomorphism ring. We introduce the following notions which will serve us for obtaining corresponding
results on (dual) self-Rickart modules.

\begin{defn} \rm A right $R$-module $M$ is called:
\begin{enumerate}
\item \emph{$k$-quasi-retractable} if ${\rm Hom}_R(M,{\rm Ker}(f))\neq 0$ for every $f\in {\rm End}_R(M)$
with ${\rm Ker}(f)\neq 0$.
\item \emph{$c$-quasi-coretractable} if ${\rm Hom}_R({\rm Coker}(f),M)\neq 0$ for every $f\in {\rm End}_R(M)$ with ${\rm
Coker}(f)\neq 0$.
\end{enumerate}
\end{defn}

For a right $R$-module $M$ with $S={\rm End}_R(M)$, denote $r_M(I)=\{x\in M\mid Ix=0\}$ for every subset $\emptyset \neq
I\subseteq S$ and $l_S(N)=\{f\in S\mid fN=0\}$ for every submodule $N$ of $M$. 

\begin{lemm} Let $M$ be a right $R$-module, and let $S={\rm End}_R(M)$. Then: 
\begin{enumerate}
\item $M$ is $k$-quasi-retractable if and only if $r_S(f)\neq 0$ for every $f\in S$ with ${\rm Ker}(f)\neq 0$.
\item $M$ is $c$-quasi-coretractable if and only if $l_S(f)\neq 0$ for every $f\in S$ with ${\rm Coker}(f)\neq 0$.
\end{enumerate}
\end{lemm}

\begin{proof} (1) Assume that $M$ is $k$-quasi-retractable. Let $f\in S$ with $K={\rm Ker}(f)\neq 0$. By hypothesis
there exists a non-zero homomorphism $g:M\to K$. If $k:K\to M$ is the inclusion homomorphism, then $0\neq kg\in S$ and
$fkg=0$. Hence $0\neq kg\in r_S(f)$. 

Conversely, assume that $r_S(f)\neq 0$ for every $f\in S$ with ${\rm Ker}(f)\neq 0$. Let $f\in S$ with $K={\rm
Ker}(f)\neq 0$. By hypothesis there exists $0\neq g\in r_S(f)$. Then $fg=0$, hence there exists a homomorphism $h:M\to
K$ such that $kh=g$, where $k:K\to M$ is the inclusion homomorphism. Note that $h\neq 0$, and so $M$ is
$k$-quasi-retractable.
\end{proof}

The following theorem gives several equivalent conditions relating the (dual) self-Rickart properties of a module and
of its endomorphism ring. They enrich existing results such as \cite[Theorem~3.9]{LRR10} and \cite[Theorem~3.5]{LRR11}.

\begin{theo} \label{t:end} Let $M$ be a right $R$-module, and let $S={\rm End}_R(M)$. 
\begin{enumerate}
\item The following are equivalent:
\begin{enumerate}[(i)] 
\item $M$ is a self-Rickart right $R$-module. 
\item $S$ is a self-Rickart right $S$-module and for every $f\in S$, ${\rm Ker}(f)$ is $M$-cyclic.
\item $S$ is a self-Rickart right $S$-module and for every $f\in S$, ${\rm Ker}(f)\in {\rm Stat}({\rm Hom}_R(M,-))$.
\item $S$ is a self-Rickart right $S$-module and for every $f\in S$, ${\rm ker}(f)$ is a locally split monomorphism.
\item $S$ is a self-Rickart right $S$-module and $M$ is $k$-quasi-retractable.
\end{enumerate}
\item The following are equivalent:
\begin{enumerate}[(i)] 
\item $M$ is a dual self-Rickart right $R$-module.
\item $S$ is a self-Rickart left $S$-module and for every $f\in S$, ${\rm Coker}(f)$ is $M$-cocyclic.
\item $S$ is a self-Rickart left $S$-module and for every $f\in S$, ${\rm Coker}(f)\in {\rm Refl}({\rm Hom}_R(-,M))$.
\item $S$ is a self-Rickart left $S$-module and for every $f\in S$, ${\rm coker}(f)$ is a locally split epimorphism.
\item $S$ is a self-Rickart left $S$-module and $M$ is $c$-quasi-coretractable.
\end{enumerate}
\end{enumerate}
\end{theo}

\begin{proof} (1) (i)$\Leftrightarrow$(ii)$\Leftrightarrow$(iii) Consider the adjoint pair of covariant functors
$(T,H)$, where $$T=-\otimes_SM:{\rm Mod}(S)\to {\rm Mod}(R),$$ $$H={\rm Hom}_R(M,-):{\rm Mod}(R)\to {\rm Mod}(S).$$ Let
$\varepsilon:TH\to 1_{{\rm Mod}(R)}$ be the counit of adjunction. We have $TH(M)\cong M$, hence $\varepsilon_M$ is an
isomorphism, and so $M\in {\rm Stat}(H)$. Now take $M=N$ in Theorem \ref{t:equiv} (1) in order to obtain the
equivalences (i)$\Leftrightarrow$(ii)$\Leftrightarrow$(iii). 

(i)$\Rightarrow$(iv) Assume that $M$ is a self-Rickart right $R$-module. By the above equivalences, $S$ is a
self-Rickart right $S$-module. Let $f\in S$ with $k={\rm ker}(f):K\to M$, and let $x\in K$. By hypothesis, $K$ is a
direct summand of $M$, hence we may consider the projection $p:M\to K$ of $M$ onto $K$. Then $p(k(x))=x$, which shows
that $k$ is a locally split monomorphism.

(iv)$\Rightarrow$(i) Assume that (iv) holds. Let $f\in S$ with $k={\rm ker}(f):K\to M$. Since $S$
is a self-Rickart right $S$-module, $r_S(f)=eS$ for some $e=e^2\in S$. Then $feS=0$, whence $eM\subseteq r_M(f)=K$. Now
let $x\in K$. By hypothesis, $k$ is a locally split monomorphism, hence there exists a homomorphism $h:M\to
K$ such that $hk(x)=x$. We have $f(h(M))=0$, hence $fh=0$, and so $hk\in r_S(f)=eS$. Then $(1-e)hk=0$, and so $hk=ehk$.
It follows that $$x=hk(x)=ehk(x)=ex\in eM,$$ hence $K\subseteq eM$. Therefore, $K=eM$ is a direct summand of $M$, which
shows that $M$ is self-Rickart.

(ii)$\Rightarrow$(v) This is clear.

(v)$\Rightarrow$(i) Assume that (v) holds. Let $f\in S$. Since $S$ is a self-Rickart right $S$-module, $r_S(f)=eS$ for
some $e=e^2\in S$. Then $feS=0$, whence $eM\subseteq r_M(f)={\rm Ker}(f)$. Now let $g=f+e\in S$. Then
$r_S(g)=r_S(f)\cap r_S(e)=eS\cap (1-e)S=0$. Since $M$ is $c$-quasi-retractable, we must have $0={\rm Ker}(g)={\rm
Ker}(f)\cap (1-e)M$, and so $M={\rm Ker}(f)\oplus (1-e)M$. Hence ${\rm Ker}(f)$ is a direct summand of $M$, which shows
that $M$ is self-Rickart.  

(2) (i)$\Leftrightarrow$(ii)$\Leftrightarrow$(iii) Consider the right adjoint pair of contravariant functors $(H,H')$,
where $$H={\rm Hom}_R(-,M):{\rm Mod}(R)\to {\rm Mod}(S^{\rm op}),$$ $$H'={\rm Hom}_S(-,M):{\rm Mod}(S^{\rm op})\to {\rm
Mod}(R).$$ Let $\eta:1_{{\rm Mod}(S^{\rm op})}\to H'H$ be the unit of adjunction. We have $H'H(M)\cong M$, hence
$\eta_M$ is an isomorphism, and so $M\in {\rm Refl}(H)$. Now take $M=N$ in Theorem \ref{t:dual} (2) in
order to obtain the equivalences (i)$\Leftrightarrow$(ii)$\Leftrightarrow$(iii). 

(i)$\Rightarrow$(iv) Assume that $M$ is a dual self-Rickart right $R$-module. By the above equivalences, $S$ is a
self-Rickart left $S$-module. Let $f\in S$ with $c={\rm coker}(f):M\to C$, and let $x\in C$. By hypothesis, $C$ is
isomorphic to a direct summand of $M$, hence we may consider the injection $j:C\to M$ of $C$ into $M$. Then
$c(j(x))=x$, which shows that $c$ is a locally split epimorphism.

(iv)$\Rightarrow$(i) Assume that (iv) holds. Let $f\in S$ with $c={\rm coker}(f):M\to C=M/fM$. Since $S$ is a
self-Rickart left $S$-module, $l_S(f)=Se$ for some $e=e^2\in S$. Then $Se\subseteq l_S(fM)$, hence 
$$fM\subseteq r_M(l_S(fM))\subseteq r_M(Se)=(1-e)M.$$ Now let $y\in M$ and denote $z=c(1-e)y\in C$. By hypothesis, $c$
is a locally split epimorphism, hence there exists a homomorphism $h:C\to M$ such that $c(h(z))=z$. Since $hcf=0$,
we have $hc\in l_S(f)=Se$, hence $hc(1-e)=0$. Then $$c(1-e)y=chc(1-e)y=0,$$ and so $(1-e)y\in {\rm Ker}(c)=fM$. 
Thus we have $(1-e)M\subseteq fM$. Therefore, $fM=(1-e)M$ is a direct summand of $M$, which shows that $M$ is dual
self-Rickart. 

(ii)$\Rightarrow$(v) This is clear.

(v)$\Rightarrow$(i) Assume that (v) holds. Let $f\in S$. Since $S$ is a self-Rickart left $S$-module, $l_S(f)=Se$ for
some $e=e^2\in S$. Then $Se\subseteq l_S(fM)$, hence $fM\subseteq (1-e)M$, as in the proof of the implication
(iv)$\Rightarrow$(i). Now let $g=f+e\in S$. Then $l_S(g)=l_S(f)\cap l_S(e)=Se\cap S(1-e)=0$. Since $M$ is
$c$-quasi-coretractable, we must have ${\rm Coker}(g)=0$, and so $M={\rm Im}(g)=fM\oplus eM$. Hence $fM$ is a direct
summand of $M$, which shows that $M$ is dual self-Rickart.
\end{proof}

Now we deduce the transfer of the (dual) Baer property to endomorphism rings of modules, completing existing results
such as \cite[Theorem~2.5]{RR09} and \cite[Theorem~3.6]{KST}. 

\begin{coll} Let $M$ be a right $R$-module, and let $S={\rm End}_R(M)$. 
\begin{enumerate}
\item The following are equivalent:
\begin{enumerate}[(i)] 
\item $M$ is a self-Baer right $R$-module. 
\item $S$ is a self-Baer right $S$-module and for every set $I$ and for every family $(f_i)_{i\in I}$ with each $f_i\in
S$, $\bigcap_{i\in I}{\rm Ker}(f_i)$ is $M$-cyclic.
\item $S$ is a self-Baer right $S$-module and for every set $I$ and for every family $(f_i)_{i\in I}$ with each $f_i\in
S$, $\bigcap_{i\in I}{\rm Ker}(f_i)\in {\rm Stat}({\rm Hom}_R(M,-))$.
\item $S$ is a self-Baer right $S$-module and for every set $I$ and for every family $(f_i)_{i\in I}$ with each $f_i\in
S$, $\bigcap_{i\in I}{\rm ker}(f_i)$ is a locally split monomorphism.
\item $S$ is a self-Baer right $S$-module and $M$ is quasi-retractable.
\end{enumerate}
\item The following are equivalent:
\begin{enumerate}[(i)] 
\item $M$ is a dual self-Baer right $R$-module.
\item $S$ is a self-Baer left $S$-module and for every set $I$ and for every family $(f_i)_{i\in I}$ with each
$f_i\in S$, $\sum_{i\in I} {\rm Im}(f_i)$ is $M$-cocyclic.
\item $S$ is a self-Baer left $S$-module and for every set $I$ and for every family $(f_i)_{i\in I}$ with each
$f_i\in S$, $\sum_{i\in I} {\rm Im}(f_i)\in {\rm Adst}({\rm Hom}_R(M,-))$.
\item $S$ is a self-Baer left $S$-module and for every set $I$ and for every family $(f_i)_{i\in I}$ with each
$f_i\in S$, $\sum_{i\in I} {\rm im}(f_i)$ is a locally split epimorphism.
\item $S$ is a self-Baer left $S$-module and $M$ is quasi-coretractable.
\end{enumerate}
\end{enumerate}
\end{coll}

\begin{proof} The equivalences (i)$\Leftrightarrow$(ii)$\Leftrightarrow$(iii) follow by Theorem \ref{t:end} and Lemmas
\ref{l:BR} and \ref{l:kerim}. The other equivalences follow in a similar way as the corresponding ones from Theorem
\ref{t:end} with endomorphisms replaced by families of endomorphisms, kernels replaced by intersections of kernels, and
cokernels replaced by sums of images. 
\end{proof}

\subsection{Graded modules}

Now let us recall some terminology and properties of graded rings and modules, following \cite{Nasta-04}. Let $G$ be a
group. A ring $R$ is called \emph{$G$-graded} if there is a family $(R_{\sigma})_{\sigma\in G}$ of additive subgroups of
$R$ such that $R=\bigoplus_{\sigma\in G}R_{\sigma}$ and $R_{\sigma}R_{\tau}\subseteq R_{\sigma\tau}$ for every
$\sigma,\tau\in G$. 

For a $G$-graded ring $R=\bigoplus_{\sigma\in G}R_{\sigma}$, denote by ${\rm gr}(R)$ the category which
has as objects the $G$-graded unital right $R$-modules and as morphisms the morphisms of $G$-graded unital right
$R$-modules, defined as follows. For a $G$-graded ring $R=\bigoplus_{\sigma\in G}R_{\sigma}$, a \emph{$G$-graded} (for
short, \emph{graded}) right $R$-module is a right $R$-module $M$ such that $M=\bigoplus_{\sigma\in G}M_{\sigma}$, where
every $M_{\sigma}$ is an additive subgroup of $M$ and for every $\lambda,\sigma\in G$, we have $M_{\sigma}\cdot
R_{\lambda}\subseteq M_{\sigma\lambda}$. For two $G$-graded right $R$-modules $M$ and $N$, the morphisms between them
are defined as follows: $$\hom{\mathrm{gr}(R)}{M}{N}=\{f\in \hom{R}{M}{N}\mid f(M_{\sigma})\subseteq N_{\sigma} \textrm{
for every } \sigma\in G\}.$$ Note that $\mathrm{gr}(R)$ is a Grothendieck category.

For $M,N\in \mathrm{gr}(R)$ we may consider the $G$-graded abelian group $\HOM{R}{M}{N}$, whose
$\sigma$-th homogeneous component is $$\HOM{R}{M}{N}_{\sigma} = \{f\in \hom{R}{M}{N} \,\mid\, f(M_{\lambda})\subseteq
M_{\sigma\lambda} \mbox{ for all }\sigma\in G \}.$$ For $M=N$, $S={\rm END}_R(M)=\HOM{R}{M}{M}$ is a $G$-graded ring
and $M$ is a graded $(S,R)$-bimodule, that is, $S_{\tau}\cdot M_{\sigma}\cdot R_{\lambda}\subseteq
M_{\tau\sigma\lambda}$ for every $\tau,\sigma,\lambda\in G$. 

For a graded right $S$-module $N$, the right $R$-module $N\otimes_SM$ may be graded by
$$(N\otimes_SM)_{\tau}=\left \{\sum_{\sigma\lambda=\tau}n_{\sigma}\otimes m_{\lambda} \mid n_{\sigma}\in
N_{\sigma}, m_{\lambda}\in M_{\lambda}\right \}.$$ Then the covariant functor $-\otimes_SM:{\rm gr}(S)\to {\rm gr}(R)$
is a left adjoint to the covariant functor $\HOM{R}{M}{-}:{\rm gr}(R)\to {\rm gr}(S)$.

For a graded right $R$-module $M$ with $S={\rm END}_R(M)$, the functors $\HOM{R}{-}{M}:{\rm gr}(R)\to {\rm gr}(S^{\rm
op})$ and $\HOM{S}{-}{M}:{\rm gr}(S^{\rm op})\to {\rm gr}(R)$ form a right adjoint pair of contravariant functors.

\begin{coll} \label{c:endgr} Let $M$ be a graded right $R$-module, and let $S={\rm END}_R(M)$. 
\begin{enumerate} \item The following are equivalent:
\begin{enumerate}[(i)]
\item $M$ is a self-Rickart graded right $R$-module.
\item $S$ is a self-Rickart graded right $S$-module and for every $f\in S$, ${\rm Ker}(f)$ is $M$-cyclic.
\item $S$ is a self-Rickart graded right $S$-module and for every $f\in S$, ${\rm Ker}(f)\in {\rm Stat}(\HOM{R}{M}{-})$.
\end{enumerate}
\item The following are equivalent:
\begin{enumerate}[(i)]
\item $M$ is a dual self-Rickart graded right $R$-module.
\item $S$ is a self-Rickart graded left $S$-module and for every $f\in S$, ${\rm Coker}(f)$ is $M$-cocyclic.
\item $S$ is a self-Rickart graded left $S$-module and for every $f\in S$, ${\rm Coker}(f)\in {\rm
Refl}(\HOM{R}{-}{M})$.
\end{enumerate}
\end{enumerate}
\end{coll}

\begin{proof} (1) Consider the adjoint pair of covariant functors $(T,H)$, where $$T=-\otimes_SM:{\rm gr}(S)\to {\rm
gr}(R),$$ $$H=\HOM{R}{M}{-}:{\rm gr}(R)\to {\rm gr}(S).$$ Let $\varepsilon:TH\to 1_{{\rm gr}(R)}$ be the counit of
adjunction. We have $TH(M)\cong M$, hence $\varepsilon_M$ is an isomorphism, and so $M\in {\rm Stat}(H)$. Now take $M=N$
in Theorem \ref{t:equiv} (1). 

(2) Consider the right adjoint pair of contravariant functors $(H,H')$, where $$H=\HOM{R}{-}{M}:{\rm gr}(R)\to
{\rm gr}(S^{\rm op}),$$ $$H'=\HOM{S}{-}{M}:{\rm gr}(S^{\rm op})\to {\rm gr}(R).$$ Let $\eta:1_{{\rm gr}(S^{\rm
op})}\to H'H$ be the unit of adjunction. We have $H'H(M)\cong M$, hence $\eta_M$ is an isomorphism, and so $M\in {\rm
Refl}(H)$. Now take $M=N$ in Theorem \ref{t:dual} (2). 
\end{proof}

Now we deduce the transfer of the Baer property to endomorphism rings of graded modules.

\begin{coll} Let $M$ be a graded right $R$-module, and let $S={\rm END}_R(M)$. 
\begin{enumerate}
\item The following are equivalent:
\begin{enumerate}[(i)] 
\item $M$ is a self-Baer graded right $R$-module. 
\item $S$ is a self-Baer graded right $S$-module and for every set $I$ and for every family $(f_i)_{i\in I}$ with each
$f_i\in S$, $\bigcap_{i\in I}{\rm Ker}(f_i)$ is $M$-cyclic.
\item $S$ is a self-Baer graded right $S$-module and for every set $I$ and for every family $(f_i)_{i\in I}$ with each
$f_i\in S$, $\bigcap_{i\in I}{\rm Ker}(f_i)\in {\rm Stat}({\rm HOM}_R(M,-))$.
\end{enumerate}
\item The following are equivalent:
\begin{enumerate}[(i)] 
\item $M$ is a dual self-Baer graded right $R$-module.
\item $S$ is a self-Baer graded left $S$-module and for every set $I$ and for every family $(f_i)_{i\in I}$ with each
$f_i\in S$, $\sum_{i\in I} {\rm Im}(f_i)$ is $M$-cocyclic.
\item $S$ is a self-Baer graded left $S$-module and for every set $I$ and for every family $(f_i)_{i\in I}$ with each
$f_i\in S$, $\sum_{i\in I} {\rm Im}(f_i)\in {\rm Refl}({\rm HOM}_R(-,M))$.
\end{enumerate}
\end{enumerate}
\end{coll}

\begin{proof} This follows by Corollary \ref{c:endgr} and Lemmas \ref{l:BR} and \ref{l:kerim}.
\end{proof}

\end{document}